\documentclass[10pt,reqno]{amsart}
\usepackage{bbm}
\usepackage{mathrsfs}
\usepackage{amsfonts} 
\usepackage[dvipsnames,usenames]{color}
\usepackage{graphicx,latexsym,bm,amsmath,amssymb,verbatim,multicol,lscape}
\usepackage{pgfplots}

\newtheorem{thm}{Theorem} [section]

\newtheorem{lem}[thm]{Lemma}

\theoremstyle{definition}
\newtheorem{defn}[thm]{Definition}
\theoremstyle{remark}
\newtheorem{rem}[thm]{Remark}

\numberwithin{equation}{section}

\begin{document}
\title{The Igusa local zeta functions of superelliptic curves}%
\author[Q.Y. Yin]{Qiuyu Yin}
\address{Mathematical College, Sichuan University, Chengdu 610064, P.R. China}
\email{yinqiuyu26@126.com}
\author[S.F. Hong]{Shaofang Hong*}
\address{Mathematical College, Sichuan University, Chengdu 610064, P.R. China}
\email{sfhong@scu.edu.cn, s-f.hong@tom.com, hongsf02@yahoo.com }
\thanks{*Hong is the corresponding author and was supported partially
by National Science Foundation of China Grant \#11771304.}

\date{\today}%
\begin{abstract}
Let $K$ be a local field and $f(x)\in K[x]$ be a non-constant
polynomial. The local zeta function $Z_f(s, \chi)$ was first
introduced by Weil, then studied in detail by Igusa. When
${\rm char}(K)=0$, Igusa proved that $Z_f(s, \chi)$ is a rational  
function of $q^{-s}$ by using the resolution of singularities.
Later on, Denef gave another proof of this remarkable result.
However, if ${\rm char}(K)>0$, the question of rationality of
$Z_f(s, \chi)$ is still kept open. Actually, there are
only a few known results so far. In this paper, we investigate
the local zeta functions of two-variable polynomial $g(x, y)$,
where $g(x, y)=0$ is the superelliptic curve with coefficients
in a non-archimedean local field of positive characteristic.
By using the notable Igusa's stationary phase formula and
with the help of some results due to Denef and
Z${\rm \acute{u}}$${\rm\tilde{n}}$iga-Galindo,
and developing a detailed analysis, we prove the rationality
of these local zeta functions and also describe explicitly
all their candidate poles.
\end{abstract}

\maketitle

\section{\bf Introduction}

Let $K$ be a local field and $f(x)\in K[x]$ be a non-constant
polynomial. The local zeta function of $f$ was first introduced by Weil,
then studied in detail by Igusa, who established many fundamental results
and posed several conjectures about these zeta functions. In this paper,
we are mainly concerned with the rationality of local zeta functions
of superelliptic curves and its poles.

Throughout, we let $K$ be a non-archimedean local field with
$\mathcal{O}_K$ as its ring of integers. Let $\mathcal{O}_K^{\times}$
be the group of units of $\mathcal{O}_K$ and $\mathcal{P}_K$ be the
unique maximal ideal of $\mathcal{O}_K$. We fix an element $\pi\in K$
such that $\mathcal{P}_K=\pi\mathcal{O}_K$.
Let $\mathbb{F}_q\cong\mathcal{O}_K/\mathcal{P}_K$ be the residue field of $K$,
which is a finite field and $q={\rm card}(\mathcal{O}_K/\mathcal{P}_K)$.
For $x\in K^{\times}$, we denote its {\it valuation} over $K$ by
${\rm ord}(x)$ such that ${\rm ord}(\pi)=1$.
Then $x$ can be written uniquely as $ac(x)\pi^{{\rm ord}(x)}$,
where $ac(x)\in\mathcal{O}_K^{\times}$ is called the
{\it angular component} of $x$, and let ${\rm ord}(0):=+\infty$. Moreover,
let $|x|_K:=|x|=q^{-{\rm ord}(x)}$ be the {\it absolute value} of $x$,
and $|dx|$ be the Haar measure on $K^n$ such that the
measure of $\mathcal{O}_K^{n}$ is one. Given a multiplicative character
$\chi: \mathcal{O}_K^{\times}\rightarrow \mathbb{C}^{\times}$ and
putting $\chi(0):=0$, then for any $f(x)\in \mathcal{O}_K[x]$ and
$s\in\mathbb{C}$ with ${\rm Re}(s)>0$, the {\it Igusa's local zeta function},
denoted by $Z_f(s, \chi)$, is defined as
$$Z_f(s, \chi):=\int_{\mathcal{O}_K^{n}}{\chi(ac f(x))|f(x)|^s|dx|}.$$

If ${\rm char}(K)=0$, that is, $K$ is a finite extension of the $p$-adic field,
Igusa \cite{[Igu1]} \cite{[Igu2]} proved that $Z_f(s, \chi)$ is a rational
function of $q^{-s}$ by using the resolution of singularities.
Later on, Denef \cite{[De1]} gave another proof of this important result.

However, when ${\rm char}(K)>0$, the question of rationality is still kept open.
Actually, there are only a few results known so far.
For example, Z${\rm \acute{u}}$${\rm\tilde{n}}$iga-Galindo \cite{[ZG2]} proved that
if $f$ is a polynomial globally non-degenerate with respect to its Newton polyhedra,
then $Z_f(s, \chi)$ is a rational function of $q^{-s}$. The basic tool he used
is called the {\it stationary phase formula} (abbreviated for SPF), which was first
introduced by Igusa \cite{[Igu3]}, then became a powerful tool in the study of
local zeta function in positive characteristics. One can consult \cite{[LIS]},
\cite{[M]} and \cite{[YH]} for more information about Igusa's local zeta functions.

In this paper, we study  the local zeta function of two-variable polynomial
$g(x, y)$, where $g(x, y)=0$ is the so-called {\it superelliptic curve},
that is
$$g(x, y)=y^m-f(x),$$
where ${\rm char}(K)\nmid m$ and $f(x)\in\mathcal{O}_K[x]$
is of degree $n$ such that $f(x)$ splits completely over $K$. Let
\begin{equation}\label{eqno 1.1}
f(x)=\gamma_0\prod_{i=1}^k(x-\gamma_i)^{n_i}
\end{equation}
be the standard factorization of $f$ over $K$ with $\sum_{i=1}^kn_i=n$.
If $m=2$ and $n=3$, then $g(x, y)=0$ is the elliptic curve. In this case,
Meuser and Robinson \cite{[MR]} studied
$Z_g(s, \chi_{\rm triv})$ and determined its explicit form. In particular,
the denominator of $Z_g(s, \chi_{\rm triv})$ is trivial. That is, the only
possible pole of $Z_g(s, \chi_{\rm triv})$ is $-1$. However, for the
general superelliptic curve $g(x, y)=0$, the rationality and candidate
poles of $Z_g(s, \chi)$ are more complicated and still unknown so far.
In the current paper, we study this problem. Actually, we will
prove the rationality of $Z_g(s, \chi)$ and list all their candidate poles.

In what follows, let $K$ be a non-archimedean local field of characteristic $p$,
and $n$ be a positive integer. For any ring $A$, one lets
$A^*:=A\setminus\{0\}$. We define the function ${\it ldeg}$ as
${\rm ldeg}(0):=+\infty$, and ${\rm ldeg}(h):=\min\{i|a_i\ne 0\}$ if
$h(x)=\sum_{i=0}^n a_ix^i\in\mathcal{O}_K[x]$ is a nonzero polynomial. Let
$$\mathcal{A}_K:=\big\{h(x)=\sum_{i=0}^n a_ix^i\in\mathcal{O}_K[x]|a_0\in\mathcal{O}_K^{\times},
a_i\in\mathcal{P}_K\ {\rm for\ any}\ i\ge 1\big\}.$$
Furthermore, if any sum or any product is empty,
we let it equal 0 or 1, respectively.

Now we state the first main result of this paper.

\begin{thm}\label{thm 1.1}
Let $h(x, y)=\mu_1x^d+\mu_2y^m+\pi h_0(x)\in\mathcal{O}_K[x, y]$,
where $\mu_1\in\mathcal{O}_K^{\times}$, $\mu_2\ne 0$ and $d$ is
a nonnegative integer and ${\rm ldeg}(h_0)\ge d+1$.
Let $m$ be an integer such that $m\ge 2$ and $p\nmid m$.
Then we have
\begin{align*}
Z_h(s, \chi)={\left\{\begin{array}{rl}
\dfrac{H_1(q^{-s})}{1-q^{-1-s}},\ \ &{\it if}\ d\le 1,\\
\dfrac{H_2(q^{-s})}{(1-q^{-1-s})(1-q^{-\tilde d-\tilde m
-\tilde d\tilde m\gcd(d, m)s})},\ \ &{\it otherwise},
\end{array}\right.}
\end{align*}
where $\tilde d=d/\gcd(d, m)$, $\tilde m=m/\gcd(d, m)$ and
$H_1(x), H_2(x)\in\mathbb{C}[x]$.
\end{thm}

Using Theorem \ref{thm 1.1}, one can prove the second main result
of this paper as follows.

\begin{thm}\label{thm 1.2}
Let $g(x, y)=y^m-f(x)$ with the factorization {\rm (\ref{eqno 1.1})}
such that $m\ge 2$ and $p\nmid m$.
Then the local zeta function $Z_g(s, \chi)$ is a rational function of $q^{-s}$.
More explicitly, we have
$$Z_g(s, \chi)=\dfrac{G(q^{-s})}{(1-q^{-1-s})
\prod\limits_{i\in T,\atop n_i\ge 2}(1-q^{-\tilde n_i-m_i-\tilde n_im_i\gcd(n_i, m)s})},$$
where $T:=\{i|1\le i\le k, \gamma_i\in\mathcal{O}_K\}$,
$\tilde n_i=n_i/\gcd(n_i, m)$, $m_i=m/\gcd(n_i, m)$ and $G(x)\in\mathbb{C}[x]$.
\end{thm}

\begin{rem}
We should point out that the polynomials treated in \cite{[ZG2]} are the
so-called non-degenerate polynomials. But the polynomials studied in
Theorems \ref{thm 1.1} and \ref{thm 1.2} may be degenerate.
For instance, let $h(x, y)=x^2(\pi x-1)^2+y^m$ with $m$ being a positive integer
such that $m\ge 2$ and $p\nmid m$. It is clear that $h$ satisfies the condition
of Theorems \ref{thm 1.1} and \ref{thm 1.2}. However, if we let $\tilde h(x):=x^2(\pi x-1)^2$,
then for any $a\in K^*$, $(\pi^{-1}, a)\in(K^*)^2$ is a singular point
of $\tilde h$. Thus $h$ is degenerate with respect to its Newton
polyhedra (see Definition 2.2 below). So one cannot make use of
Theorem A of \cite{[ZG2]} directly.
\end{rem}

The paper is organized as follows. In Section 2,
we first review some known results on Newton polyhedra
and state a result due to Z${\rm \acute{u}}$${\rm\tilde{n}}$iga-Galindo.
Then we introduce the SPF and use it to prove a lemma (Lemma \ref{lem 2.6}
below) which will be used in what follows. Moreover, we show
several other lemmas needed in the proofs of our main results. 
In concluding Section 2, we show a result (Lemma \ref{lem 2.11} below) 
which plays an important role in the proof of Theorem \ref{thm 1.1}. 
Finally, in Sections 3 and 4, we use the lemmas presented 
in previous section to give the proofs of Theorems \ref{thm 1.1} 
and \ref{thm 1.2}, respectively.

\section{\bf Preliminaries and lemmas }

\subsection{\bf Newton polyhedra and Z${\rm \acute{u}}$${\rm\tilde{n}}$iga-Galindo's theorem}

We begin with the definition of Newton polyhedra for two-variable polynomials.

Let $\mathbb{R}_{+}:=\{x\in \mathbb{R}|x\geq 0\}$ and
$f(x)=\sum_{l}{a_lx^l}\in K[x]$ be a two-variable polynomial
satisfying $f(0)=0$ with the notation
$$a_lx^l=a_{l_1, l_2}x_1^{l_1}x_2^{l_2}\ \ {\rm for\ any} \ l=(l_1, l_2)\in\mathbb{N}^2.$$
The {\it support set} of $f$ is denoted by
${\rm supp}(f):=\{l\in \mathbb{N}^2|a_l\neq 0\}.$
Then we define the {\it Newton polyhedra} $\Gamma(f)$ of $f$
to be the convex hull in $\mathbb{R}_{+}^2$ of the set
$\bigcup_{l\in {\rm supp}(f)}{(l+\mathbb{R}_{+}^2)}.$

We call $\gamma$ a proper {\it face} of $\Gamma(f)$ if $\gamma$ is a non-empty convex set
which is obtained by intersecting $\Gamma(f)$ with an affine hyperplane $H$,
such that $\Gamma(f)$ is contained in one of two half-plane determined by $H$.
The hyperplane $H$ is called the {\it supporting hyperplane} of $\gamma$. Let
$a_{\gamma}=(a_1, a_2)\in \mathbb{N}^2\setminus\{0\}$
be the vector which is perpendicular to $H$ and let $|a_{\gamma}|:=a_1+a_2$.
A face $\gamma$ of codimension one is named {\it facet}.

Let $\langle , \rangle$ denote the usual {\it inner product} of $\mathbb{R}^2$.
For any $a\in \mathbb{R}^2_{+}$, let
$$m(a):=\inf_{b\in \Gamma(f)}{\{\langle a, b\rangle \}},$$
and for any $a\in \mathbb{R}^2_{+}\setminus \{0\}$, {\it the first meet locus} of $a$
is denoted by $F(a)$ defined as
$$F(a):=\{x\in \Gamma(f)|\langle a, x\rangle=m(a)\}.$$
In fact, $F(a)$ is a proper face of $\Gamma(f)$. Moreover, there exists an
equivalence relation $\simeq$ on $\mathbb{R}^2_{+}\setminus \{0\}$:
For any $a, \tilde a\in \mathbb{R}^2_{+}\setminus \{0\}$, we have that
$a\simeq \tilde a \ {\rm if\ and\ only\ if}\ F(a)=F(\tilde a).$
Furthermore, let $\gamma$ be a proper face of $\Gamma(f)$,
we define {\it the cone associated to $\gamma$} as
$$\Delta_{\gamma}:=\{a\in \mathbb{R}^2_{+}\setminus \{0\}|F(a)=\gamma\}.$$
It is obvious that $\Delta_{\gamma}\cap\Delta_{\gamma^{'}}=\emptyset$
for different proper faces $\gamma, \gamma^{'}$ of $\Gamma(f)$.
Thus one has the following partition of $\mathbb{R}^2_{+}$:
\begin{equation*}
\mathbb{R}^2_{+}=\big\{0\big\}\bigcup\big(\bigcup_{\gamma}\Delta_{\gamma}\big),
\end{equation*}
where $\gamma$ runs over all proper faces of $\Gamma(f)$.
Then it follows that
\begin{equation}\label{eqno 2.1}
\mathbb{N}^2=\big\{0\big\}\bigcup\Big(\bigcup_{\gamma}
\big(\Delta_{\gamma}\bigcap (\mathbb{N}^2\setminus \{0\})\big)\Big).
\end{equation}
Let $C$ be any set with $C\subseteq\mathbb{N}^2$, we define the set
$E(C)$ associated to $C$ as
$$E(C):=\big\{(x_1, x_2)\in \mathcal{O}_K^{2}|
\big({\rm ord}(x_1), {\rm ord}(x_2)\big)\in C\big\}.$$
Then by (\ref{eqno 2.1}), one has
\begin{align}\label{2.2}
Z_f(s, \chi)=Z_f\big(s, \chi, (\mathcal{O}_K^{\times})^2\big)
+\sum_{\gamma}Z_f\Big(s, \chi, E\big(\Delta_{\gamma}\bigcap (\mathbb{N}^2\setminus\{0\})\big)\Big),
\end{align}
where $\gamma$ runs over all the proper faces of $\Gamma(f)$.
The following lemma is due to Denef and describes the structure of $\Delta_{\gamma}$.

\begin{lem}\label{lem 2.1}\cite{[De2]}
Let $\gamma$ be a proper face of $\Gamma(f)$, and $\omega_1, \omega_2,\cdots, \omega_e$
be the facets of $\Gamma(f)$ which contain $\gamma$.
Denote by $\alpha_1, \alpha_2,\cdots, \alpha_e$ the vectors which
are perpendicular to $\omega_1, \omega_2,\cdots, \omega_e$, respectively.
Then $\Delta_{\gamma}=\big\{\sum_{i=1}^ea_i\alpha_i|a_i\in \mathbb{R}^{+}\big\}$.
\end{lem}

Let $h(x, y)=\mu_1x^d+\mu_2y^m+\pi h_0(x)$ with $d\ge 1$ and ${\rm ldeg}(h_0)\ge d+1$.
First one can easily derive that
\begin{equation*}
\Gamma(h)=\big\{(x, y)|x\ge d, y\ge m, mx+dy\ge dm\big\}.
\end{equation*}
Moreover, $\Gamma(h)$ has exact five proper faces, that is,
\begin{align}\label{2.3}
&\gamma_1=\{(x, 0)|x\ge d\},\ \ \gamma_2=\{(0, y)|y\ge m\},\nonumber\\
&\gamma_3=\big\{(x, y)|mx+dy=dm,\ 0\le x\le d, 0\le y\le m\big\},\nonumber\\
&\gamma_4=(d, 0),\ \ \gamma_5=(0, m).
\end{align}
For facets $\gamma_1$, $\gamma_2$ and $\gamma_3$, we choose
\begin{align}\label{2.3'}
\alpha_1=(0, 1), \alpha_2=(1, 0) \ {\rm and} \ \alpha_3=(\tilde m, \tilde d)
\end{align}
to be the vectors which are perpendicular to $\gamma_1$, $\gamma_2$ and $\gamma_3$,
respectively, where $\tilde d=d/\gcd(d, m)$ and $\tilde m=m/\gcd(d, m)$.
Then Lemma \ref{lem 2.1} gives us that $\Delta_{\gamma_i}=\{a\alpha_i|a\in\mathbb{R}^{+}\}$
for $i=1, 2, 3$ and
$$
\Delta_{\gamma_4}=\{a\alpha_1+b\alpha_3|a, b\in\mathbb{R}^{+}\},\ \ \ \
\Delta_{\gamma_5}=\{a\alpha_2+b\alpha_3|a, b\in\mathbb{R}^{+}\}.
$$
Since $\gcd(\tilde d, \tilde m)=1$, so for $i=1, 2, 3$, we have
\begin{equation}\label{eqno 2.4}
\Delta_{\gamma_i}\bigcap (\mathbb{N}^2\setminus \{0\})
=\{a\alpha_i|a\in\mathbb{Z}^+\}.
\end{equation}
Moreover, for $i=4, 5$, let $S_4=\{1,3\}$ and $S_5=\{2, 3\}$. Then
\begin{equation}\label{eqno 2.5}
\Delta_{\gamma_i}\bigcap (\mathbb{N}^2\setminus \{0\})
=\bigcup\limits_{c\in \tilde S_i}
\Big\{c+\sum_{j\in S_i}a_j\alpha_j|a_j\in\mathbb{Z}^+\Big\},
\end{equation}
where
$$\tilde S_i=\mathbb{N}^2\bigcap\Big\{\sum_{j\in S_i}\lambda_j\alpha_j|0\le \lambda_j<1\Big\}.$$

Now we introduce a well-known definition (see, for example \cite{[ZG2]}).

\begin{defn}
A polynomial $f(x)=\sum_{i}{a_i}x^i\in K[x]$ is called {\it globally non-degenerate
with respect to its Newton polyhedra $\Gamma(f)$} if it satisfies the following
two properties:

(1). The origin of $K^n$ is a singular point of $f(x)$. Namely, one has
$$f(0,...,0)=\dfrac{\partial f}{\partial x_1}(0,...,0)=\cdots=\dfrac{\partial f}{\partial x_n}(0,...,0)=0.$$

(2). For every face $\gamma\subset\Gamma(f)$ (including $\Gamma(f)$ itself),
the polynomial
$$f_{\gamma}(x):=\sum_{i\in \gamma}{a_i}x^i$$
has the property that there is no $x\in (K^*)^n$ such that
$x$ is a singular point of $f_{\gamma}$.
\end{defn}

As a conclusion of this subsection, we state a result of
Z${\rm \acute{u}}$${\rm\tilde{n}}$iga-Galindo as follows.


\begin{lem} {\rm\cite{[ZG2]}}\label{lem 2.3}
Let K be a non-archimedean local field, and let $f(x)\in \mathcal{O}_K[x]$
be a polynomial globally non-degenerate with respect to its Newton polyhedra
$\Gamma(f)$. Then the Igusa's local zeta function $Z_f(s, \chi)$ is a rational
function of $q^{-s}$. Furthermore, if $s$ is a pole of $Z_f(s, \chi)$, then
$$s=-\dfrac{|a_{\gamma}|}{m(a_{\gamma})}
+\dfrac{2\pi i}{\log q}\dfrac{k}{m(a_{\gamma})},\ k\in \mathbb{Z}$$
for some facet $\gamma$ of $\Gamma(f)$ with perpendicular $a_{\gamma}$ if
$m(a_{\gamma})\neq 0$, and
$$s=-1+\dfrac{2\pi i}{\log q}k,\ k\in \mathbb{Z}$$
otherwise.
\end{lem}

\subsection{\bf Some lemmas}

In this subsection, we present some lemmas which will be used later.
At first, we recall the so-called SPF.
For any $x\in\mathcal{O}_K^n$, let $\bar{x}$ be the image of $x$ under
the canonical homomorphism
$\mathcal{O}_K^n\rightarrow (\mathcal{O}_K/\mathcal{P}_K)^n\cong\mathbb{F}_q^n$.
For $f(x)\in \mathcal{O}_K[x]$, $\bar{f} (x)$ stands for the polynomial
obtained by reducing modulo $\pi$ the coefficients of $f(x)$.
Let $A$ be any ring and $f(x)\in A[x]$. We define $V_f(A):=\{x\in A^n|f(x)=0\}$.
By ${\rm Sing}_f(A)$ we denote the set of $A$-value singular points of $V_f$,
namely,
$${\rm Sing}_f(A):=\Big\{x\in A^n\Big|f(x)
=\dfrac{\partial f}{\partial x_1}(x)=\cdots
=\dfrac{\partial f}{\partial x_n}(x)=0\Big\}.$$

We fix a lifting $R$ of $\mathbb{F}_q$ in $\mathcal{O}_K$. That is,
the set $R^n$ is mapped bijectively onto $\mathbb{F}_q^n$ by the
canonical homomorphism.
Let $\bar{D}$ be a subset of $\mathbb{F}_q^n$ and $D$ be its preimage
under the canonical homomorphism. We also denote by $S(f, D)$ the subset
of $R^n$ mapped bijectively to the set ${\rm Sing}_{\bar{f}}(\mathbb{F}_q)\cap \bar{D}$.
If $D=\mathcal{O}_K^n$, we use notation $S(f)$ instead of $S(f, \mathcal{O}_K^n)$.
Furthermore, we denote
\begin{align*}
v(\bar{f}, D, \chi):={\left\{\begin{array}{rl}
q^{-n} \cdot \#\{\bar{P}\in \bar{D}|\bar{P}\notin V_{\bar{f}}(\mathbb{F}_q)\}, \ \ \ \
&{\rm if} \ \chi=\chi_{{\rm triv}},\\
q^{-nc_{\chi}}\sum\limits_{\{P\in D|\bar{P}\notin V_{\bar{f}}(\mathbb{F}_q)\}
{\rm mod}\ \mathcal{P}_K^{c_{\chi}}}{\chi(ac f(P)),}\  \ \ \
&{\rm otherwise},
\end{array}\right.}
\end{align*}
where $c_{\chi}$ is the conductor of $\chi$, and
\begin{align*}
\sigma(\bar{f}, D, \chi):={\left\{\begin{array}{rl}
q^{-n}\cdot \#\{\bar{P}\in \bar{D}|\bar{P}\ {\rm is\ a\ nonsingular\ point\ of} \
V_{\bar{f}}(\mathbb{F}_q)\}, \ \  &{\rm if} \ \chi=\chi_{{\rm triv},}\\
0,\ \   & {\rm otherwise}.
\end{array}\right.}
\end{align*}
If $D=\mathcal{O}_K^n$, we write $v(f, \chi)$ and $\sigma(f, \chi)$
for simplicity. Finally, let
$$Z_f(s, \chi, D):=\int_D\chi(acf(x))|f(x)|^s|dx|.$$

Now we can state the SPF in the following form.

\begin{lem}\cite{[Igu3]} \cite{[YH]}\label{lem 2.4}
For any complex number $s$ with ${\rm Re}(s)>0$, we have
\begin{align*}
Z_f(s, \chi, D)=v(\bar{f}, D, \chi)&+\sigma(\bar{f}, D, \chi)
\dfrac{(1-q^{-1})q^{-s}}{1-q^{-1-s}}+Z_f(s, \chi, D_{S(f, D)}),
\end{align*}
where $D_{S(f, D)}:=\bigcup_{P\in S(f, D)}D_P$ with
$D_P:=\{x\in\mathcal{O}_K^n|x-P\in\mathcal{P}_K^n\}$.
That is, $D_{S(f, D)}$ is the preimage of
${\rm Sing}_{\bar{f}}(\mathbb{F}_q)\cap \bar{D}$
under the canonical homomorphism
$\mathcal{O}_K^n\rightarrow (\mathcal{O}_K/\mathcal{P}_K)^n$.
\end{lem}

We will make frequent use of the following facts in the remaining part of the paper.

\begin{lem}\label{lem 2.5}
For any subset $D\subseteq \mathcal{O}_K^n$, each of the following is true.

{\rm (i).} Let $a$ be any nonnegative integer. Then for $\beta\in\pi^aD$,
we have $\pi^{-a}\beta\in D$.

{\rm (ii).} Let $f(x)\in\mathcal{O}_K[x]$.
Then for any $\alpha\in\mathcal{O}_K^*$, one has
\begin{equation*}
Z_{\alpha f}(s, \chi, D)=\chi\Big(\frac{\alpha}{\pi ^{{\rm ord}(\alpha)}}\Big)
q^{-{\rm ord}(\alpha)s}Z_f(s, \chi, D).
\end{equation*}
In particular, if $\alpha=\pi^e$ with $e\in\mathbb{N}$, then
$Z_{\pi^ef}(s, \chi, D)=q^{-es}Z_f(s, \chi, D)$.
\end{lem}

\begin{proof} Part (i) is clear true. In the following we show part (ii).
Since $ac$ is a multiplicative function,  we derive that
\begin{align}\label{2.10'}
Z_{\alpha f}(s, \chi, D)=&\int_D{\chi(ac(\alpha f(x)))|\alpha f(x)|^s|dx|}\nonumber\\
=&\chi(ac(\alpha))|\alpha|^s\int_D{\chi(ac f(x))|f(x)|^s|dx|}\nonumber\\
=&\chi\Big(\frac{\alpha}{\pi ^{{\rm ord}(\alpha)}}\Big)q^{-{\rm ord}(\alpha)s}Z_f(s, \chi, D)
\end{align}
as expected.

Moreover, let $\alpha=\pi^e$ with $e\in\mathbb{N}$.
Since $\chi\big(\frac{\alpha}{\pi ^{{\rm ord}(\alpha)}}\big)=\chi(1)=1$ and ord$(\pi^{e})=e$,
(\ref{2.10'}) implies that $Z_{\pi^ef}(s, \chi, D)=q^{-es}Z_f(s, \chi, D)$.
Hence part (ii) is proved.

This concludes the proof of Lemma \ref{lem 2.5}.
\end{proof}

\begin{lem}\label{lem 2.6}
Let $b, c\in\mathcal{O}_K^*$ (recall that $\mathcal{O}_K^*=\mathcal{O}_K\setminus\{0\}$)
and $b_i\in\mathcal{O}_K$ for all integers
$i$ with $1\le i\le n$. Let $l(x, y)=bx^{d_0}\prod_{i=1}^n(x-b_i)^{d_i}+cy^m$.
Let $b=\pi^{e_b}b_0$ with $b_0\in\mathcal{O}_K^{\times}$. Then
\begin{equation}\label{eqno 2.9}
Z_l(s, \chi)=\dfrac{L_1(q^{-s})}{1-q^{-1-s}}+L_2(q^{-s})Z_v(s, \chi),
\end{equation}
where $v(x, y)=b_0x^{d_0}\prod_{i=1}^n(x-b_i)^{d_i}+wy^m$
with $w\in\mathcal{O}_K^*$ and $L_1(x), L_2[x]\in\mathbb{C}[x]$.
Furthermore, if $b_i\in\mathcal{P}_K\setminus\{0\}$
for all integers $i$ with $1\le i\le n$, then
\begin{equation}\label{eqno 2.10}
Z_v(s, \chi)=\dfrac{V_1(q^{-s})}{1-q^{-1-s}}+V_2(q^{-s})Z_{\tilde v}(s, \chi),
\end{equation}
where $\tilde v(x, y)=b_0x^{d_0}\prod_{i=1}^n(x-\tilde{b}_i)^{d_i}+\tilde wy^m$
satisfies that $\tilde w\in\mathcal{O}_K^*$ and
$\tilde{b}_i\in\mathcal{O}_K^{\times}$ for at least
one index $i$ with $1\le i\le n$, and $V_1(x), V_2[x]\in\mathbb{C}[x]$.
\end{lem}

\begin{proof}
Let $c=\pi^{e_c}c_1$ with $c_1\in\mathcal{O}_K^{\times}$ and $e_c\ge 0$
being an integer. We prove (\ref{eqno 2.9}) by considering the following two cases.

{\sc Case 1.} $e_b\le e_c$. Then by Lemma \ref{lem 2.5} (ii), one has
$Z_l(s, \chi)=q^{-e_bs}Z_{l_1}(s, \chi)$
with $l_1(x, y)=b_0x^{d_0}\prod_{i=1}^n(x-b_i)^{d_i}+\pi^{e_c-e_b}c_1y^m$.
So (\ref{eqno 2.9}) is true in this case.

{\sc Case 2.} $e_b>e_c$. Then Lemma \ref{lem 2.5} (ii) gives us that
\begin{equation}\label{eqno 2.11}
Z_l(s, \chi)=q^{-e_cs}Z_{\ell}(s, \chi),
\end{equation}
where $\ell (x, y)=b_0\pi^{e_b-e_c}x^{d_0}\prod_{i=1}^n(x-b_i)^{d_i}+c_1y^m$.
Since $e_b-e_c>0$ and $c_1\in\mathcal{O}_K^{\times}$, one has
$\bar{\ell}(x, y)=\bar{c}_1y^m$
with $\bar{c}_1\ne\bar{0}$. Then it is easy to see that
$${\rm Sing}_{\bar{\ell}}(\mathbb{F}_q)\bigcap \mathbb{F}_q^2
=\{(x, y)\in \mathbb{F}_q^2|y=\bar{0} \}.$$
Thus $S(\ell)=\{(x, y)\in R^2|y=0\}$ and $D_{S(\ell)}=\mathcal{O}_K\times\pi\mathcal{O}_K$.
By Lemma \ref{lem 2.4}, we deduce that
\begin{align}\label{2.12}
Z_{\ell}(s, \chi)=&v(\bar{\ell}, \chi)+\sigma(\bar{\ell}, \chi)
\dfrac{(1-q^{-1})q^{-s}}{1-q^{-1-s}}+Z_{\ell}(s, \chi, D_{S(\ell)})\nonumber\\
=&\dfrac{L_{2,1}(q^{-s})}{1-q^{-1-s}}+Z_{\ell}(s, \chi, D_{S(\ell)})
\end{align}
where $L_{2,1}(x)\in\mathbb{C}[x]$ since $v(\bar{\ell}, \chi)$ and $\sigma(\bar{\ell}, \chi)$
are constants defined in Lemma \ref{lem 2.4}.

For $Z_{\ell}(s, \chi, D_{S(\ell)})$, we make the change of variables
of the form: $(x, y)\mapsto (x_1, \pi y_1)$.
Then by Lemma \ref{lem 2.5} (i), one has
\begin{align}\label{eqno 2.13}
Z_{\ell}(s, \chi, D_{S(\ell})=q^{-1}Z_{\ell_1}(s, \chi),
\end{align}
where $\ell_1(x, y):=b_0\pi^{e_b-e_c}x^{d_0}\prod_{i=1}^n(x-b_i)^{d_i}+c_1\pi^my^m$.
By (\ref{eqno 2.11}) to (\ref{eqno 2.13}), we obtain that
\begin{equation}\label{eqno 2.14}
Z_l(s, \chi)=\dfrac{q^{-e_cs}L_{2, 1}(q^{-s})}{1-q^{-1-s}}+q^{-1-e_cs}Z_{\ell_1}(s, \chi).
\end{equation}

{\sc Subcase 2.1.} $e_b-e_c\le m$. Then Lemma \ref{lem 2.5} (ii) yields that
\begin{equation}\label{eqno 2.15}
Z_{\ell_1}(s, \chi)=q^{-(e_b-e_c)s}Z_{\ell_2}(s, \chi),
\end{equation}
where $\ell_2(x, y)=b_0x^{d_0}\prod_{i=1}^n(x-b_i)^{d_i}+c_1\pi^{m-(e_b-e_c)}y^m$.
By (\ref{eqno 2.14}) and (\ref{eqno 2.15}), we have
$$Z_l(s, \chi)=\dfrac{q^{-e_cs}L_{2, 1}(q^{-s})}{1-q^{-1-s}}+q^{-1-e_bs}Z_{\ell_2}(s, \chi)$$
as (\ref{eqno 2.9}) expected.

{\sc Subcase 2.2.} $e_b-e_c>m$. Let $e_b-e_c=tm+r$ with $0\le r<m$.
By applying Lemma \ref{lem 2.4} for $t$ times to $\ell_1$,
the above argument together with (\ref{eqno 2.14}) gives us that
\begin{equation}\label{eqno 2.16}
Z_l(s, \chi)=\dfrac{L_{2, 2}(q^{-s})}{1-q^{-1-s}}+q^{-(t+1)-(tm+e_c)s}Z_{\ell_3}(s, \chi),
\end{equation}
where $\ell_3(x, y)=b_0\pi^rx^{d_0}\prod_{i=1}^n(x-b_i)^{d_i}+c_1\pi^my^m$
and $L_{2,2}(x)\in\mathbb{C}[x]$. But $r<m$.
Thus by Lemma \ref{lem 2.5} (ii), we have
\begin{equation}\label{eqno 2.17}
Z_{\ell_3}(s, \chi)=q^{-rs}Z_{\ell_4}(s, \chi),
\end{equation}
where $\ell_4(x, y)=b_0x^{d_0}\prod_{i=1}^n(x-b_i)^{d_i}+c_1\pi^{m-r}y^m$.
Putting (\ref{eqno 2.17}) into (\ref{eqno 2.16}), we arrive at
\begin{align*}
Z_l(s, \chi)=&\dfrac{L_{2, 2}(q^{-s})}{1-q^{-1-s}}
+q^{-(t+1)-(tm+r+e_c)s}Z_{\ell_4}(s, \chi)\\
=&\dfrac{L_{2, 2}(q^{-s})}{1-q^{-1-s}}
+q^{-(t+1)-e_bs}Z_{\ell_4}(s, \chi).
\end{align*}
Thus (\ref{eqno 2.9}) holds for Case 2. This completes the proof of (\ref{eqno 2.9}).

In the remaining part of the proof, we show (\ref{eqno 2.10}).
For any integer $i$ with $1\le i\le n$, let $b_i=\pi^{e_i}b_{i, 1}$ with
$b_{i, 1}\in\mathcal{O}_K^{\times}$. Then $b_i\in\mathcal{P}_K$ implies that $e_i\ge 1$.
Let $\sum_{i=0}^nd_i=d$. Then $\bar v(x, y)=\bar{b}_0x^d+\bar{w}y^m$ with
$\bar{b}_0\ne\bar 0$ since $b_0\in\mathcal{O}_K^{\times}$.
On the other hand, since $\mathcal{O}_K=\mathcal{O}_K^{\times}\cup\pi\mathcal{O}_K$,
we have
\begin{equation}\label{eqno 2.18}
Z_v(s, \chi)=Z_v(s, \chi, D_1)+Z_v(s, \chi, D_2),
\end{equation}
where $D_1=\mathcal{O}_K^{\times}\times\mathcal{O}_K$ and
$D_2=\pi\mathcal{O}_K\times\mathcal{O}_K$.

For $Z_v(s, \chi, D_1)$, one has $\bar{D}_1=\mathbb{F}_q^*\times\mathbb{F}_q$.
If $P=(x_0, y_0)\in {\rm Sing}_{\bar v}(\mathbb{F}_q)$, then
$$\bar v(P)=\frac{\partial \bar v}{\partial x}(P)
=\frac{\partial \bar v}{\partial y}(P)=\bar{0}.$$
Since $m\ge 2$ and $p\nmid m$,
$\frac{\partial \bar v}{\partial y}(P)=m\bar{w}y_0^{m-1}=\bar{0}$
gives us that $\bar wy_0=\bar{0}$. But $\bar{v}(P)=\bar{0}$ yields that $x_0=0$.
It then follows that ${\rm Sing}_{\bar{v}}(\mathbb{F}_q)\bigcap \bar{D}_1=\emptyset$,
which implies that $S(v, D_1)=\emptyset$. Using Lemma \ref{lem 2.4}, we have
\begin{equation}\label{eqno 2.19}
Z_v(s, \chi, D_1)=\dfrac{V_{1, 1}(q^{-s})}{1-q^{-1-s}}
\end{equation}
where $V_{1, 1}(x)\in\mathbb{C}[x]$.

For $Z_v(s, \chi, D_2)$, we make the change of variables of the form:
$(x, y)\mapsto (\pi x_1, y_1)$, then Lemma \ref{lem 2.5} (i) tells us that
\begin{equation*}
Z_v(s, \chi, D_2)=q^{-1}Z_{v_1}(s, \chi),
\end{equation*}
where
$$v_1(x, y)=b_0\pi^{d_0}x^{d_0}\prod_{i=1}^n(\pi x-b_i)^{d_i}+wy^m
:=b_0\pi^dx^{d_0}\prod_{i=1}^n(x-b_{i, 2})^{d_i}+wy^m$$
with $b_{i, 2}=\pi^{-1}b_i=\pi^{e_i-1}b_{i, 1}\in\mathcal{O}_K$
since $e_i\ge 1$. By (\ref{eqno 2.9}) applied to $v_1$ gives us that
\begin{equation}\label{eqno 2.20}
Z_v(s, \chi, D_2)=q^{-1}Z_{v_1}(s, \chi)
=\dfrac{V_{2, 1}(q^{-s})}{1-q^{-1-s}}+V_{2, 2}(q^{-s})Z_{v_2}(s, \chi),
\end{equation}
where $v_2(x, y):=b_0x^{d_0}\prod_{i=1}^n(x-b_{i, 2})^{d_i}+w_1y^m$ with
$w_1\in \mathcal{O}_K^*$ and $V_{2, 1}(x), V_{2, 2}(x)\in\mathbb{C}[x]$.
From (\ref{eqno 2.18}) to (\ref{eqno 2.20}), we obtain that
\begin{equation}\label{eqno 2.21}
Z_v(s, \chi)=\dfrac{V_1(q^{-s})}{1-q^{-1-s}}+V_2(q^{-s})Z_{v_2}(s, \chi),
\end{equation}
where $V_1(x), V_2(x)\in\mathbb{C}[x]$.

If $\min\{e_1,\cdots, e_n\}=1$, there exists one integer $i_0$ with $1\le i_0\le n$
such that $e_{i_0}=1$, which implies that $b_{i_0, 2}\in\mathcal{O}_K^{\times}$
since ${\rm ord}(b_{i_0, 2})={\rm ord}(\pi^{e_{i_0-1}}b_{i_0, 1})=e_{i_0}-1=0$.
Hence (\ref{eqno 2.10}) is true in this case.

If $\min\{e_1,\cdots, e_n\}>1$, applying the above argument to $v_2$
for $\min\{e_1,\cdots, e_n\}-1$ times, then (\ref{eqno 2.21}) yields that
\begin{equation*}
Z_v(s, \chi)=\dfrac{V_3(q^{-s})}{1-q^{-1-s}}+V_4(q^{-s})Z_{v_3}(s, \chi),
\end{equation*}
where $V_3(x), V_4(x)\in\mathbb{C}[x]$ and
$$v_3(x, y):=b_0x^{d_0}\prod_{i=1}^n(x-b_{i, 3})^{d_i}+w_2y^m$$
with $b_{i, 3}=\pi^{e_i-\min\{e_1,\cdots, e_n\}}b_{i, 1}$ and $w_2\in\mathcal{O}_K^*$.
Moreover, there exists one integer $j_0$ with $1\le j_0\le n$ such that
$e_{j_0}=\min\{e_1,\cdots, e_n\}$, i.e. $b_{j_0, 3}\in\mathcal{O}_K^{\times}$.
Thus (\ref{eqno 2.10}) holds in this case. So (\ref{eqno 2.10}) is proved.
This finishes the proof of Lemma \ref{lem 2.6}.
\end{proof}

\begin{lem}\label{lem 2.7}
Let $u(x)=u_0\prod_{i=1}^n(x-a_i)\in\mathcal{O}_K[x]$
with $u_0\ne 0$ and $a_i\in K$ for all $i$. Let
$T:=\{1\le i\le n | {\rm ord}(a_i)\ge 0\}$. Then
${\rm ord}(u_0)+\sum_{i\notin T}{\rm ord}(a_i)\ge 0$.
\end{lem}

\begin{proof}
Let $\tilde T:=\{1,\cdots, n\}\setminus T$, i.e. $\tilde T=\{1\le i\le n | {\rm ord}(a_i)<0\}$.
If $\tilde T=\emptyset$, then Lemma 2.5 is trivial since $u_0\in\mathcal{O}_K$,
which implies that ${\rm ord}(u_0)\ge 0$. If $\tilde T\ne\emptyset$,
then one lets $|\tilde T|=k$. Thus $1\le k\le n$. Without loss of any generality,
one may let $\tilde T=\{1,\cdots, k\}$. Write $u(x)=\sum_{j=0}^nu_jx^{n-j}$.
We can derive that
\begin{align}\label{2.6}
u_k=&(-1)^ku_0\sum_{1\le i_1<\cdots<i_k\le n}a_{i_i}\cdots a_{i_k}\nonumber\\
=&(-1)^ku_0\big(a_1\cdots a_k+\sum_{1\le i_1<\cdots<i_k\le n, \atop i_k>k}
a_{i_i}\cdots a_{i_k}\big)\nonumber\\
:=&(-1)^ku_0A.
\end{align}

Since for any integer $i_k$ with $i_k>k$, i.e. $i_k\in T$,
one has ${\rm ord}(\alpha_{i_k})\ge 0$. It then follows that
\begin{equation}\label{eqno 2.7}
{\rm ord}(a_{i_1}\cdots a_{i_k})=\sum_{j=1}^k{\rm ord}(a_{i_j})
>\sum_{j=1}^k{\rm ord}(a_j)={\rm ord}(a_1\cdots a_k).
\end{equation}
Since $K$ is non-archimedean, (\ref{eqno 2.7}) implies that
$${\rm ord}\big(\sum_{1\le i_1<\cdots<i_k\le n, \atop i_k>k}
a_{i_i}\cdots a_{i_k}\big)\ge\min_{1\le i_1<\cdots<i_k\le n, \atop i_k>k}
\{{\rm ord}(a_{i_i}\cdots a_{i_k})\}>{\rm ord}(a_1\cdots a_k).$$
Hence by the isosceles triangle principle (see, for instance, \cite{[K]}), we have
\begin{equation}\label{eqno 2.8}
{\rm ord}(A)={\rm ord}(a_1\cdots a_k)=\sum_{i=1}^k{\rm ord}(a_i).
\end{equation}
Thus by (\ref{2.6}) and (\ref{eqno 2.8}), one gets that
\begin{align*}
{\rm ord}(u_k)={\rm ord}((-1)^ku_0A)=&{\rm ord}((-1)^k)+{\rm ord}(u_0)+{\rm ord}(A)\\
=&{\rm ord}(u_0)+\sum_{i=1}^k{\rm ord}(a_i).
\end{align*}
But $u_k\in\mathcal{O}_K$ tells us that  ${\rm ord}(u_k)\ge 0$. So the
desired result follows immediately. Thus Lemma \ref{lem 2.7} is proved.
\end{proof}

\begin{lem}\label{lem 2.8}
Let $\mathbb{F}_q$ be the finite field of characteristic $p$.
Let $\imath(x, y)=a+by^m\in\mathbb{F}_q[x, y]$ be any polynomial satisfying that
$a\in\mathbb{F}_q^*$ and $m$ is an integer with $m\ge 2$ and $p\nmid m$.
Then ${\rm Sing}_{\imath}(\mathbb{F}_q)=\emptyset$.
\end{lem}

\begin{proof}
Let $P=(x_0, y_0)$ be any element in $\mathbb{F}_q^2$.

If $b=0$, then
$\imath(P)=a\ne 0$ since $a\in\mathbb{F}_q^*$,
which infers that $P\notin{\rm Sing}_{\imath}(\mathbb{F}_q)$.

If $b\ne 0$, then $\frac{\partial \imath}{\partial y}(P)=mby_0^{m-1}=0$
only if $y_0=0$ since $p\nmid m$. But $y_0=0$ implies that
$\imath(P)=a\ne 0$. This tells us that $P\notin{\rm Sing}_{\imath}(\mathbb{F}_q)$.

In conclusion, we have ${\rm Sing}_{\imath}(\mathbb{F}_q)=\emptyset$
as one expects. So Lemma \ref{lem 2.8} is proved.
\end{proof}

Now we state a useful definition introduced by Z${\rm\acute{u}}$${\rm\tilde{n}}$iga-Galindo.

\begin{defn}
Let $f(x)\in\mathcal{O}_K[x]$ and $P\in \mathcal{O}_K^n$ such that
$P\notin {\rm Sing}_f(\mathcal{O}_K)$. We define the {\it index} $L(f, P)$ by
$$L(f, P):=\min\Big\{{\rm ord}\big(f(P)\big), {\rm ord}\big(\dfrac{\partial f}{\partial x_1}(P)\big),
\cdots, {\rm ord}\big(\dfrac{\partial f}{\partial x_n}(P)\big)\Big\}.$$
\end{defn}
In \cite{[ZG2]}, Z${\rm\acute{u}}$${\rm\tilde{n}}$iga-Galindo
proved that for any $f(x)\in\mathcal{O}_K[x]$ such that
${\rm Sing}_f(\mathcal{O}_K)\cap(\mathcal{O}_K^{\times})^n=\emptyset$,
$L(f, P)$ is bounded by a constant depended only on $f$ for all $P\in(\mathcal{O}_K^{\times})^n$.
Let $C(f, (\mathcal{O}_K^{\times})^n)$ be the minimal constant with
$L(f, P)\le C(f, (\mathcal{O}_K^{\times})^n)$ and $C(f, (\mathcal{O}_K^{\times})^n)\in\mathbb{N}$.
The following result is due to Z${\rm\acute{u}}$${\rm\tilde{n}}$iga-Galindo
and is in fact a special case of Corollary 2.5 of \cite{[ZG2]}.

\begin{lem}\label{lem 2.10} \cite{[ZG2]}
Let $F(x)=f(x)+\pi^\beta g(x)\in\mathcal{O}_K[x]$ such that
$\beta\ge C(f, (\mathcal{O}_K^{\times})^n)+1$ and
$${\rm Sing}_F(\mathcal{O}_K)\bigcap(\mathcal{O}_K^{\times})^n
={\rm Sing}_f(\mathcal{O}_K)\bigcap(\mathcal{O}_K^{\times})^n=\emptyset.$$
Then $Z_F(s, \chi, (\mathcal{O}_K^{\times})^n)=Z_f(s, \chi, (\mathcal{O}_K^{\times})^n).$
\end{lem}

From Lemma \ref{lem 2.10}, we can derive the following result.

\begin{lem}\label{lem 2.11}
Let $r(x, y)=r_1(x, y)+\pi r_0(x)\in\mathcal{O}_K[x, y]$
and $L(r_1, P)=0$ for all $P\in(\mathcal{O}_K^{\times})^2$.
Then $Z_r(s, \chi, (\mathcal{O}_K^{\times})^2)=Z_{r_1}(s, \chi, (\mathcal{O}_K^{\times})^2).$
\end{lem}

\begin{proof}
Since $L(r_1, P)=0$ for all $P\in(\mathcal{O}_K^{\times})^2$, it then follows that
$C(r_1, (\mathcal{O}_K^{\times})^2)=0$. Further, the hypothesis
$L(r_1, P)=0$ tells us that at least one of
${\rm ord}\big(r_1(P)\big), {\rm ord}\big(\frac{\partial r_1}{\partial x}(P)\big)$
and ${\rm ord}\big(\frac{\partial r_1}{\partial y}(P)\big)$ is equal to 0.
Equivalently, at least one of $r_1(P)$, $\frac{\partial r_1}{\partial x}(P)$
and $\frac{\partial r_1}{\partial y}(P)$ belongs $\mathcal{O}_K^{\times}$,
namely, at least one of them is nonzero. This implies that
${\rm Sing}_{r_1}(\mathcal{O}_K)\bigcap(\mathcal{O}_K^{\times})^2=\emptyset$.

If ${\rm ord}(r_1(P))=0$, then one has
$${\rm ord}(r(P))={\rm ord}(r_1(P)+\pi r_0(P))={\rm ord}(r_1(P))=0.$$
Hence $r(P)\in\mathcal{O}_K^{\times}$, i.e. $r(P)\ne 0$.

If ${\rm ord}(\frac{\partial r_1}{\partial x}(P))=0$ or
${\rm ord}(\frac{\partial r_1}{\partial y}(P))=0$,
then the same argument as above yields that $\frac{\partial r}{\partial x}(P)\ne 0$
or $\frac{\partial r}{\partial y}(P)\ne 0$. It follows immediately that
$${\rm Sing}_r(\mathcal{O}_K)\bigcap(\mathcal{O}_K^{\times})^2
={\rm Sing}_{r_1}(\mathcal{O}_K)\bigcap(\mathcal{O}_K^{\times})^2=\emptyset.$$
Then applying Lemma \ref{lem 2.10} gives us that
$$Z_r(s, \chi, (\mathcal{O}_K^{\times})^2)
=Z_{r_1}(s, \chi, (\mathcal{O}_K^{\times})^2)$$
as expected. This ends the proof of Lemma \ref{lem 2.11}.
\end{proof}

\section{\bf Proof of Theorem \ref{thm 1.1}}

In this section, we show Theorem \ref{thm 1.1}.\\

{\it Proof of Theorem \ref{thm 1.1}.}
Since $\mu_1\in\mathcal{O}_K^{\times}$, one has
$\bar{h}(x, y)=\bar{\mu}_1x^d+\bar{\mu}_2y^m$ with $\bar{\mu}_1\ne\bar{0}$.
Moreover, since $p\nmid m$, there exists an integer $m_0$ such that $mm_0\equiv 1\pmod p$,
which implies that $mm_0=1$ in $K$, i.e. $m\in\mathcal{O}_K^{\times}$.
Let $\mu_2=\pi^{e_0}\mu_{2, 1}$ with $\mu_{2, 1}\in\mathcal{O}_K^{\times}$.

We notice that if $P=(x_0, y_0)\in {\rm Sing}_{\bar{h}}(\mathbb{F}_q)$,
then $\frac{\partial \bar{g}}{\partial y}(P)=m\bar{\mu}_2y_0^{m-1}=\bar{0}$
tells us that $\bar{\mu}_2y_0=\bar{0}$, which infers that
$\bar h(P)=\bar \mu _1x_0^d$. So if $d=0$, then $\bar{h}(P)=\bar{\mu}_1\ne\bar{0}$,
and if $d=1$, then $\frac{\partial \bar{h}}{\partial x}(P)=\bar{\mu}_1\ne\bar{0}$.
Thus ${\rm Sing}_{\bar{h}}(\mathbb{F}_q)=\emptyset$ if $d\le 1$. This yields
that $S(h)=\emptyset$. Using Lemma \ref{lem 2.4}, one gets that
$$Z_h(s, \chi)=\dfrac{H_1(q^{-s})}{1-q^{-1-s}},$$
where $H_1(x)\in\mathbb{C}[x]$. Hence Theorem \ref{thm 1.1} is true if $d\le 1$.

In what follows, we let $d\ge 2$. Let $h(x, y)=h_1(x, y)+\pi h_0(x)$ with
$h_1(x, y)=\mu_1x^d+\mu_2y^m$. Evidently, $h_1(x, y)=\mu_1x^d+\mu_2y^m$
is a polynomial globally non-degenerate with respect to its Newton polyhedra
since $p\nmid m$, but $h$ may be degenerate with respect to its Newton polyhedra.
That is, Lemma \ref{lem 2.3} can be applied to $h_1(x, y)$, but cannot
be applied to $h(x, y)$ directly. If we can show that
\begin{align}\label{eqno 3.0}
Z_h(s, \chi)=Z_{h_1}(s, \chi),
\end{align}
then applying Lemma \ref{lem 2.3} to $h_1$ gives us that
$$Z_h\big(s, \chi)=Z_{h_1}\big(s, \chi)=\dfrac{H_2(q^{-s})}
{(1-q^{-1-s})(1-q^{-\tilde d-\tilde m-\tilde d\tilde m\gcd(d, m)s})},$$
where $H_2(x)\in\mathbb{C}[x]$. Thus Theorem \ref{thm 1.1} holds for $d\ge 2$.
It remains to prove that (\ref{eqno 3.0}) is true which will be
done in the following.

First, it is easy to see that $\Gamma(h)=\Gamma(h_1)$ since ${\rm ldeg}(h_0)\ge d+1$.
By (\ref{2.2}), we have
\begin{align}\label{eqno 3.1}
Z_h(s, \chi)=Z_h\big(s, \chi, (\mathcal{O}_K^{\times})^2\big)
+\sum_{i=1}^{5}Z_h\big(s, \chi, E(\Delta_{\gamma_i}\bigcap(\mathbb{N}^2\setminus \{0\}))\big),
\end{align}
where $\gamma_i$ is defined in (\ref{2.3}). Now, we calculate the
six integrals on the right-hand side of (\ref{eqno 3.1}) respectively.

For $Z_h(s, \chi, (\mathcal{O}_K^{\times})^2)$, let $P=(x_0, y_0)\in(\mathcal{O}_K^{\times})^2$.
If $\mu_2\in\mathcal{O}_K^{\times}$, then
$\frac{\partial h_1}{\partial y}(P)=m\mu_2y_0^{m-1}\in\mathcal{O}_K^{\times}$
since $m\in\mathcal{O}_K^{\times}$, which implies that
${\rm ord}(\frac{\partial h_1}{\partial y}(P))=0$.
If $\mu_2\in\pi\mathcal{O}_K$, then
$${\rm ord}(h_1(P))={\rm ord}(\mu_1x_0^d+\mu_2y_0^m)={\rm ord}(\mu_1x_0^d)=0$$
since $x_0, \mu_1\in\mathcal{O}_K^{\times}$.
Thus for any $P\in(\mathcal{O}_K^{\times})^2$, one has
$$L(h_1, P)=\min\Big\{{\rm ord}\big(h_1(P)\big), {\rm ord}\big(\dfrac{\partial h_1}{\partial x}(P)\big),
{\rm ord}\big(\dfrac{\partial h_1}{\partial y}(P)\big)\Big\}=0.$$
By Lemma \ref{lem 2.11}, we have
\begin{equation}\label{eqno 3.2}
Z_h\big(s, \chi, (\mathcal{O}_K^{\times})^2)=Z_{h_1}\big(s, \chi, (\mathcal{O}_K^{\times})^2).
\end{equation}

Let $i$ be an integer with $1\le i\le 3$.
Let $\alpha_i$ be the vector given  in (\ref{2.3'}) and write
$\alpha_i=(\alpha_{i,1}, \alpha_{i, 2})$. By (\ref{eqno 2.4}), we deduce that
\begin{equation*}
E(\Delta_{\gamma_i}\bigcap(\mathbb{N}^2\setminus \{0\}))
=\bigcup_{a=1}^{\infty}(\pi^{a\alpha_{i, 1}}\mathcal{O}_K^{\times}
\times\pi^{a\alpha_{i, 2}}\mathcal{O}_K^{\times})
:=\bigcup_{a=1}^{\infty}D_i(a).
\end{equation*}
Thus
$$Z_h\big(s, \chi, E(\Delta_{\gamma_i}\bigcap(\mathbb{N}^2\setminus \{0\}))\big)
=\sum_{a=1}^{\infty}Z_h(s, \chi, D_i(a)).$$

For $Z_h(s, \chi, D_i(a))$, we make the following change of variables of the form:
$(x, y)\mapsto(\pi^{a\alpha_{i, 1}}x_1, \pi^{a\alpha_{i, 2}}y_1)$.
Then Lemma \ref{lem 2.5} (i) gives us that
\begin{equation}\label{eqno 3.3'}
Z_h\big(s, \chi, E(\Delta_{\gamma_i}\bigcap(\mathbb{N}^2\setminus \{0\}))\big)
=\sum_{a=1}^{\infty}q^{-a(\alpha_{i ,1}+\alpha_{i ,2})}Z_{h_{i, a}}(s, \chi, (\mathcal{O}_K^{\times})^2),
\end{equation}
where $h_{i, a}(x, y)=\mu_1\pi^{ad\alpha_{i, 1}}x^d+\mu_{2, 1}\pi^{e_0+am\alpha_{i, 2}}y^m
+\pi h_0(\pi^{a\alpha_{i, 1}}x)$. Let
$$e_{i, a}:=\min\{ad\alpha_{i, 1},\ e_0+am\alpha_{i, 2}\}.$$
Since ${\rm ldeg}(h_0)\ge d+1$, one can write
$h_0(x)=\sum_{j=d+1}^{\infty}b_jx^j\in\mathcal{O}_K[x]$.
Then $h_0(\pi^{a\alpha_{i, 1}}x)=\sum_{j=d+1}^{\infty}b_j\pi^{a\alpha_{i, 1}j}x^j$
with $b_j\in\mathcal{O}_K$. Thus for any integer $j$ with $j\ge d+1$,
one derives that
\begin{align*}
{\rm ord}(b_j\pi^{a\alpha_{i, 1}j})={\rm ord}(b_j)+a\alpha_{i, 1}j
>ad\alpha_{i, 1}\ge e_{i, a}.
\end{align*}
Hence it follows that $\pi^{-e_{i, a}}h_0(\pi^{a\alpha_{i, 1}}x)\in\mathcal{O}_K[x]$.
Then by Lemma \ref{lem 2.5} (ii) and (\ref{eqno 3.3'}), we derive that
\begin{align}\label{3.3}
Z_h\big(s, \chi, E(\Delta_{\gamma_i}\bigcap(\mathbb{N}^2\setminus \{0\}))\big)
=\sum_{a=1}^{\infty}q^{-a(\alpha_{i ,1}+\alpha_{i ,2})-e_{i, a}s}
Z_{\tilde h_{i, a}}(s, \chi, (\mathcal{O}_K^{\times})^2),
\end{align}
where
$\tilde h_{i, a}(x, y):=\mu_{i, a, 1}x^d+\mu_{i, a, 2}y^m+\pi(\pi^{-e_{i, a}}h_0(\pi^{a\alpha_{i, 1}}x))$
with $\mu_{i, a, 1}=\mu_1\pi^{ad\alpha_{i, 1}-e_{i, a}}$ and
$\mu_{i, a, 2}=\mu_{2, 1}\pi^{e_0+am\alpha_{i, 2}-e_{i, a}}$
satisfies that either $\mu_{i, a, 1}\in\mathcal{O}_K^{\times}$ or $\mu_{i, a, 2}\in\mathcal{O}_K^{\times}$
since $\mu_1, \mu_{2, 1}\in\mathcal{O}_K^{\times}$.

Let $r_{i, a}(x, y)=\mu_{i, a, 1}x^d+\mu_{i, a, 2}y^m$. We claim that
$L(r_{i, a}, P)=0$ for any $P\in(\mathcal{O}_K^{\times})^2$.
In fact, let $P=(x_0, y_0)\in(\mathcal{O}_K^{\times})^2$.
If $\mu_{i, a, 2}\in\mathcal{O}_K^{\times}$, then
$\frac{\partial h_1}{\partial y}(P)=m\mu_{i, a, 2}y_0^{m-1}\in\mathcal{O}_K^{\times}$
since $m\in\mathcal{O}_K^{\times}$, which tells us that
${\rm ord}(\frac{\partial h_1}{\partial y}(P))=0$. This infers that $L(r_{i, a}, P)=0$.
If $\mu_{i, a, 2}\in\pi\mathcal{O}_K$, then $\mu_{i, a, 1}\in\mathcal{O}_K^{\times}$.
Thus the discussion for $h_1$ yields that $L(r_{i, a}, P)=0$.
The claim is proved. Now by the claim and Lemma \ref{lem 2.11}, we arrive at
\begin{equation}\label{eqno 3.4}
Z_{\tilde h_{i, a}}(s, \chi, (\mathcal{O}_K^{\times})^2)=
Z_{r_{i, a}}(s, \chi, (\mathcal{O}_K^{\times})^2).
\end{equation}
Putting (\ref{eqno 3.4}) into (\ref{3.3}) gives us that
$$Z_h\big(s, \chi, E(\Delta_{\gamma_i}\bigcap(\mathbb{N}^2\setminus \{0\}))\big)
=\sum_{a=1}^{\infty}q^{-a(\alpha_{i ,1}+\alpha_{i ,2})-e_{i, a}s}
Z_{r_{i, a}}(s, \chi, (\mathcal{O}_K^{\times})^2).$$

On the other hand, the same argument as above yields that
$$Z_{h_1}\big(s, \chi, E(\Delta_{\gamma_i}\bigcap(\mathbb{N}^2\setminus \{0\}))\big)
=\sum_{a=1}^{\infty}q^{-a(\alpha_{i ,1}+\alpha_{i ,2})-e_{i, a}s}
Z_{r_{i, a}}(s, \chi, (\mathcal{O}_K^{\times})^2).$$
It then follows that for any integer $i$ with $1\le i\le 3$, we have
\begin{equation}\label{eqno 3.5}
Z_h\big(s, \chi, E(\Delta_{\gamma_i}\bigcap(\mathbb{N}^2\setminus \{0\}))\big)
=Z_{h_1}\big(s, \chi, E(\Delta_{\gamma_i}\bigcap(\mathbb{N}^2\setminus \{0\}))\big).
\end{equation}

In what follows, we show that (3.6) still keeps valid if $i=4$ and 5.

Likewise, by (\ref{eqno 2.5}), we can derive that
\begin{equation*}
E(\Delta_{\gamma_4}\bigcap(\mathbb{N}^2\setminus \{0\}))
=\bigcup_{c\in \tilde S_4}\bigcup_{a=1}^{\infty}\bigcup_{b=1}^{\infty}
(\pi^{c_1+a\alpha_{1, 1}+b\alpha_{3, 1}}\mathcal{O}_K^{\times}\times
\pi^{c_2+a\alpha_{1, 2}+b\alpha_{3, 2}}\mathcal{O}_K^{\times}),
\end{equation*}
where $c=(c_1, c_2)$. Then we get that
$$Z_h\big(s, \chi, E(\Delta_{\gamma_4}\bigcap(\mathbb{N}^2\setminus \{0\}))\big)
=\sum_{c\in \tilde S_4}\sum_{a=1}^{\infty}\sum_{b=1}^{\infty}Z_h(s, \chi, D_4(a, b, c)),$$
where
$D_4(a, b, c)=\pi^{c_1+a\alpha_{1, 1}+b\alpha_{3, 1}}\mathcal{O}_K^{\times}\times
\pi^{c_2+a\alpha_{1, 2}+b\alpha_{3, 2}}\mathcal{O}_K^{\times}$.

For $Z_h(s, \chi, D_4(a, b, c))$, let
$e_1(a, b, c)=c_1+a\alpha_{1, 1}+b\alpha_{3, 1}$
and $e_2(a, b, c)=c_2+a\alpha_{1, 2}+b\alpha_{3, 2}$.
By making the change of variables of the form:
$(x, y)\mapsto(\pi^{e_1(a, b, c)}x_1, \pi^{e_2(a, b, c)}y_1)$
and Lemma \ref{lem 2.5} (i), one has
\begin{equation}\label{eqno 3.7'}
Z_h\big(s, \chi, E(\Delta_{\gamma_4}\bigcap(\mathbb{N}^2\setminus \{0\}))\big)
=\sum_{c\in \tilde S_4}\sum_{a=1}^{\infty}\sum_{b=1}^{\infty}
q^{-e_1(a, b, c)-e_2(a, b, c)}Z_{h_{4, a, b, c}}(s, \chi, (\mathcal{O}_K^{\times})^2),
\end{equation}
where
$$h_{4, a, b, c}(x, y)=\mu_1\pi^{de_1(a, b, c)}x^d+\mu_{2, 1}\pi^{e_0+me_2(a, b, c)}y^m
+\pi h_0(\pi^{e_1(a, b, c)}x).$$
Let $e(a, b, c):=\min\{de_1(a, b, c),\ e_0+me_2(a, b, c)\}$.
The same argument gives us that
$\pi^{-e(a, b, c)}h_0(\pi^{e_1(a, b, c)}x)\in\mathcal{O}_K[x]$.
Then Lemma \ref{lem 2.5} (i) and (\ref{eqno 3.7'}) yields that
\begin{align}\label{3.6}
&Z_h\big(s, \chi, E(\Delta_{\gamma_4}\bigcap(\mathbb{N}^2\setminus \{0\}))\big)\nonumber\\
=&\sum_{c\in \tilde S_4}\sum_{a=1}^{\infty}\sum_{b=1}^{\infty}
q^{-e_1(a, b, c)-e_2(a, b, c)-e(a, b, c)s}
Z_{\tilde h_{4, a, b, c}}(s, \chi, (\mathcal{O}_K^{\times})^2),
\end{align}
where
\begin{align*}
\tilde h_{4, a, b, c}(x, y):=&\mu_1\pi^{de_1(a, b, c)-e(a, b, c)}x^d
+\mu_2\pi^{me_2(a, b, c)-e(a, b, c)}y^m+\pi^{1-e(a, b, c)}h_0(\pi^{e_1(a, b, c)}x)\\
:=&\mu_{4, a, b, c}x^d+\tilde\mu_{4, a, b, c}y^m+\pi(\pi^{-e(a, b, c)}h_0(\pi^{e_1(a, b, c)}x)).
\end{align*}
Let $r_{4, a, b, c}(x, y)=\mu_{4, a, b, c}x^d+\tilde\mu_{4, a, b, c}y^m$.
As before, one can deduce that $L(r_{4, a, b, c}, P)=0$
for all $P\in(\mathcal{O}_K^{\times})^2$.
Hence Lemma \ref{lem 2.11} infers that
\begin{equation}\label{eqno 3.7}
Z_{\tilde h_{4, a, b, c}}(s, \chi, (\mathcal{O}_K^{\times})^2)=
Z_{r_{4, a, b, c}}(s, \chi, (\mathcal{O}_K^{\times})^2).
\end{equation}
Therefore, combining (\ref{3.6}) with (\ref{eqno 3.7}) gives us that
\begin{align}\label{3.8}
&Z_h\big(s, \chi, E(\Delta_{\gamma_4}\bigcap(\mathbb{N}^2\setminus \{0\}))\big)\nonumber\\
=&\sum_{c\in \tilde S_4}\sum_{a=1}^{\infty}\sum_{b=1}^{\infty}
q^{-e_1(a, b, c)-e_2(a, b, c)-e(a, b, c)s}
Z_{r_{4, a, b, c}}(s, \chi, (\mathcal{O}_K^{\times})^2)\nonumber\\
=&Z_{h_1}\big(s, \chi, S(\Delta_{\gamma_4}\bigcap(\mathbb{N}^2\setminus \{0\}))\big).
\end{align}

Furthermore, by the similar argument as for the case $i=4$,
we obtain that
\begin{equation}\label{eqno 3.9}
Z_h\big(s, \chi, E(\Delta_{\gamma_5}\bigcap(\mathbb{N}^2\setminus \{0\}))\big)=
Z_{h_1}\big(s, \chi, E(\Delta_{\gamma_5}\bigcap(\mathbb{N}^2\setminus \{0\}))\big).
\end{equation}

Finally, by (\ref{eqno 3.1}), (\ref{eqno 3.2}), (\ref{eqno 3.5}),
(\ref{3.8}) and (\ref{eqno 3.9}), we have
$$Z_h\big(s, \chi)=Z_{h_1}\big(s, \chi)$$
as desired. So (\ref{eqno 3.0}) is proved.
This concludes the proof of Theorem \ref{thm 1.1}.
\hfill$\Box$

\section{\bf Proof of Theorem \ref{thm 1.2}}

In this final section, we supply the proof of Theorem \ref{thm 1.2}.\\

{\it Proof of Theorem \ref{thm 1.2}.}
Let $g(x, y)=y^m-f(x)$ with the factorization (\ref{eqno 1.1}), where $m$
is an integer such that $m\ge 2$ and $p\nmid m$.
For any integer $i$ with $0\le i\le n$, let $\gamma_i=\pi^{e_i}\gamma_{i, 1}$
with $\gamma_{i, 1}\in\mathcal{O}_K^{\times}$. Then $e_i={\rm ord}(\gamma_i)$
for all integers $i$ with $0\le i\le k$.
Now we define the set $T$ of indexes by
$T:=\{1\le i\le k | \gamma_i\in\mathcal{O}_K\}=\{1\le i\le k | e_i\ge 0\}$.
Then by Lemma \ref{lem 2.7}, one has $e_0+\sum_{i\notin T}n_ie_i\ge 0$.
It follows that
\begin{align*}
g(x, y)=&y^m-\gamma_0\Big(\prod_{i\in T}(x-\gamma_i)^{n_i}\Big)\Big(\prod_{i\notin T}(x-\gamma_i)^{n_i}\Big)\\
=&y^m-\pi^{e_0+\sum_{i\notin T}n_ie_i}\gamma_{0, 1}f_1(x)\prod_{i\in T}(x-\gamma_i)^{n_i},
\end{align*}
where
$$f_1(x):=\prod_{i\notin T}(\pi^{-e_i}x-\gamma_{i, 1})^{n_i}\in\mathcal{A}_K$$
since $e_i<0$ and $\gamma_{i, 1}\in\mathcal{O}_K^{\times}$ for any $i\notin T$,
with $\mathcal{A}_K$ being defined in the introduction section.

Using Lemma \ref{lem 2.6}, we have
\begin{equation}\label{eqno 4.1}
Z_g(s, \chi)=\dfrac{F_1(q^{-s})}{1-q^{-1-s}}+F_2(q^{-s})Z_{g_1}(s, \chi),
\end{equation}
where
$$g_1(x, y):=\sigma f_1(x)\prod_{i\in T}(x-\gamma_i)^{n_i}+\delta y^m$$
with $\sigma:=-\gamma_{0, 1}\in\mathcal{O}_K^{\times}$,
$\delta\in\mathcal{O}_K^*$ and $F_1(x), F_2(x)\in\mathbb{C}[x]$.

If $T=\emptyset$, then $\bar g_1(x, y)=\bar\sigma\bar{\lambda}_0+\bar\delta y^m$
with $\lambda_0=f_1(0)$ satisfying that $\bar\sigma\bar{\lambda}_0\ne\bar{0}$
since $f_1(x)\in\mathcal{A}_K$. By Lemma \ref{lem 2.8}, we have
${\rm Sing}_{\bar g_1}(\mathbb{F}_q)=\emptyset$, which implies that
$S(g_1)=\emptyset$. Then Lemma \ref{lem 2.4} tells us that
\begin{equation}\label{eqno 4.2}
Z_{g_1}(s, \chi)=\dfrac{G_{1, 1}(q^{-s})}{1-q^{-1-s}}
\end{equation}
with $G_{1, 1}(x)\in\mathbb{C}[x]$. From (\ref{eqno 4.1}) and (\ref{eqno 4.2}),
we derive that
$$Z_g(s, \chi)=\dfrac{G_1(q^{-s})}{1-q^{-1-s}},$$
where $G_1(x)\in\mathbb{C}[x]$. So Theorem \ref{thm 1.2} is true if $T=\emptyset$.
It remains to treat the case $T\ne\emptyset$.

In what follows, we let $T\ne\emptyset$. Notice that if the following is true:
\begin{equation}\label{eqno 4.3}
Z_{g_1}(s, \chi)=\dfrac{\tilde G_2(q^{-s})}{(1-q^{-1-s})
\prod\limits_{i\in T\atop n_i\ge 2}(1-q^{-\tilde n_i-m_i-\tilde n_im_i\gcd(n_i, m)s})},
\end{equation}
where $\tilde G_2(x)\in\mathbb{C}[x]$, then (\ref{eqno 4.1}) together with
(\ref{eqno 4.3}) will imply the truth of Theorem \ref{thm 1.2}. So we need only
to prove that (\ref{eqno 4.3}) holds that will be done in the following.

Without loss of any generality, we may let $T:=\{1,\cdots, l\}$ with $1\le l\le k$.
Then there exists a positive integer $r$ and a strictly increasing sequence
$\{i_j\}_{j=0}^r$ of nonnegative integers with $i_0=0$ and $i_r=l$ such that
$$T=\bigcup\limits_{j=0}^{r-1}T_j,$$
where for each integer $j$ with $0\le j\le r-1$, we have
$$T_j:=\{i_j+1,\cdots, i_{j+1}\}$$
and $\bar{\gamma}_{j_1}=\bar{\gamma}_{j_2}$ if
$j_1\in T_j$ and $j_2\in T_j$, and $\bar{\gamma}_{j_1}\ne\bar{\gamma}_{j_2}$
if exactly one of $j_1$ and $j_2$ is in the set $T_j$.

Define $R_1:=\{\gamma_{i_1},\cdots, \gamma_{i_r}\}$.
Then we can choose a lifting $R$ of $\mathbb{F}_q$ in $\mathcal{O}_K$
such that $R_1\subseteq R$, and let $R_2=R\setminus R_1$.
Now we prove (\ref{eqno 4.3}) by induction on $l=|T|$.

If $l=1$, then by making the change of variables of the form:
$(x, y)\mapsto(x_1+\gamma_1, y_1)$, one has
$$Z_{g_1}(s, \chi)=Z_{\tilde g_1}(s, \chi),$$
where $\tilde g_1(x, y)=\sigma x^{n_1}f_1(x+\gamma_1)+\delta y^m$.
Since $f_1(x)\in\mathcal{A}_K$, we have $f_1(x+\gamma_1)=\pi\tilde f_1(x)+f_1(\gamma_1)$
with ${\rm ldeg}(\tilde f_1)\ge 1$. But the definition of $f_1(x)$ gives us that
for any $\alpha\in\mathcal{O}_K$, we have
\begin{equation}\label{eqno 4.4'}
f_1(\alpha)=\prod_{i\notin T}(\pi^{-e_i}\alpha-\gamma_{i, 1})^{n_i}
\equiv \prod_{i\notin T}(-\gamma_{i, 1})^{n_i}\not\equiv 0\pmod{\pi},
\end{equation}
i.e., $f_1(\alpha)\in\mathcal{O}_K^{\times}$.
Particularly, $f_1(\gamma_1)\in\mathcal{O}_K^{\times}$. It then follows that
$$\tilde g_1(x, y)=\varepsilon x^{n_1}+\delta y^m+\pi\sigma\tilde f_1(x)x^{n_1}$$
with $\varepsilon=\sigma f_1(\gamma_1)\in\mathcal{O}_K^{\times}$
and ${\rm ldeg}(\tilde f_1(x)x^{n_1})={\rm ldeg}(\tilde f_1)+n_1\ge n_1+1$.
Now with Theorem \ref{thm 1.1} applied to $\tilde g_1$, we obtain that
$$Z_{g_1}\big(s, \chi)=Z_{\tilde g_1}\big(s, \chi)=\dfrac{G_0(q^{-s})}
{(1-q^{-1-s})(1-q^{-\tilde n_1-m_1-\tilde n_1m_1\gcd(n_1, m)s})}$$
where $G_0(x)\in\mathbb{C}[x]$, as desired. So (4.3) is true if $l=1$.

In what follows, we let $t$ be a positive integer with $2\le t\le k$.
We assume that (\ref{eqno 4.3}) is true for any integer $l$ with $1\le l<t$.
Now let $l=t$. Since
$$\mathcal{O}_K=\bigcup_{a\in R}(a+\pi\mathcal{O}_K),$$
we deduce that
\begin{align}\label{4.4}
Z_{g_1}(s, \chi)=&\sum_{a\in R}Z_{g_1}(s, \chi, D_a)\nonumber\\
=&\sum_{a\in R_1}Z_{g_1}(s, \chi, D_a)+\sum_{a\in R_2}Z_{g_1}(s, \chi, D_a)\nonumber\\
=&\sum_{j=0}^{r-1}Z_{g_1}(s, \chi, D_{\gamma_{i_{j+1}}})+\sum_{a\in R_2}Z_{g_1}(s, \chi, D_a),
\end{align}
where $D_a=(a+\pi\mathcal{O}_K)\times\mathcal{O}_K$.

Let $a\in R_2$ (if $R_2$ is nonempty). Then for $Z_{g_1}(s, \chi, D_a)$,
we make the following change of variables of the form: $(x, y)\mapsto(a+\pi x_1, y_1)$.
Then $Z_{g_1}(s, \chi, D_a)=q^{-1}Z_{g_{1, a}}(s, \chi)$, where
$$g_{1, a}(x, y):=\sigma f_1(a+\pi x)\prod_{i=1}^t\big(\pi x-(\gamma_i-a)\big)^{n_i}+\beta y^m.$$
Since $a\in R_2$, one has $\gamma_i-a\in\mathcal{O}_K^{\times}$
for any integer $i$ with $1\le i\le t$, which infers that
$\bar g_{1, a}(x, y)=\bar\lambda_a+\bar\delta y^m$ with
$$\bar\lambda_a=\bar\sigma\prod_{i=1}^t\big(-(\bar\gamma_i
-\bar a)\big)^{n_i}\overline{f_1(a)}\ne\bar 0$$
since $\overline{f_1(a)}\ne\bar 0$ by (\ref{eqno 4.4'}).
Then Lemma \ref{lem 2.8} applied to $\bar g_{1, a}$ yields that
${\rm Sing}_{\bar g_{1, a}}(\mathbb{F}_q)=\emptyset$,
which implies that $S(g_{1, a})=\emptyset$.
Using Lemma \ref{lem 2.4}, we know that if $a\in R_2$, then
\begin{equation}\label{eqno 4.5}
Z_{g_1}(s, \chi, D_a)=q^{-1}Z_{g_{1, a}}(s, \chi)=\dfrac{G_{2, a}(q^{-s})}{1-q^{-1-s}},
\end{equation}
where $G_{2, a}(x)\in\mathbb{C}[x]$.

For any integer $j$ with $0\le j\le r-1$, we make the change of variables
of the form: $(x, y)\mapsto(\gamma_{i_{j+1}}+\pi x_1, y_1)$. Then
$Z_{g_1}(s, \chi, D_{\gamma_{i_{j+1}}})=q^{-1}Z_{g_{1, j}}(s, \chi),$
where
\begin{align*}
g_{1, j}(x, y)=&\sigma f_1(\gamma_{i_{j+1}}+\pi x)
\prod_{i=1}^t\big(\pi x-(\gamma_i-\gamma_{i_{j+1}})\big)^{n_i}+\delta y^m\\
=&\sigma f_1(\gamma_{i_{j+1}}+\pi x)\prod_{i\in T_j}\big(\pi x-(\gamma_i-\gamma_{i_{j+1}})\big)^{n_i}
\cdot \prod_{i\notin T_j}\big(\pi x-(\gamma_i-\gamma_{i_{j+1}})\big)^{n_i}+\delta y^m\\
:=&\sigma\pi^{\sum_{i\in T_j}n_i}f_{1, j}(x)\prod_{i\in T_j}
\big(x-\pi^{-1}(\gamma_i-\gamma_{i_{j+1}})\big)^{n_i}+\delta y^m
\end{align*}
with
$$f_{1, j}(x)=f_1(\gamma_{i_{j+1}}+\pi x)
\prod_{i\notin T_j}\big(\pi x-(\gamma_i-\gamma_{i_{j+1}})\big)^{n_i}.$$
Since $\gamma_i\not \equiv \gamma_{i_{j+1}}\pmod{\pi}$ for any
integer $i$ with $1\le i\le t$ and $i\notin T_j$, one deduces that
$$f_{1, j}(0)=f_1(\gamma_{i_{j+1}})\prod_{i\notin T_j}\big(-(\gamma_i-\gamma_{i_{j+1}})\big)^{n_i}
\in\mathcal{O}_K^{\times}$$
since $f_1(\gamma_{i_{j+1}})\in \mathcal{O}_K^{\times}$ by (\ref{eqno 4.4'}).
This yields that $f_{1, j}(x)\in\mathcal{A}_K$.
Using Lemma \ref{lem 2.6}, we arrive at
\begin{equation}\label{eqno 4.6}
Z_{g_1}(s, \chi, D_{\gamma_{i_{j+1}}})=q^{-1}Z_{g_{1, j}}(s, \chi)
=\dfrac{M_{1, j}(q^{-s})}{1-q^{-1-s}}+M_{2, j}(q^{-s})Z_{g_{2, j}}(s, \chi),
\end{equation}
where
$$g_{2, j}(x, y):=\sigma f_{1, j}(x)\prod_{i\in T_j}
\big(x-\pi^{-1}(\gamma_i-\gamma_{i_{j+1}})\big)^{n_i}+\delta_1 y^m$$
with $\delta_1\in\mathcal{O}_K^*$ and $M_{1, j}(x), M_{2, j}\in\mathbb{C}[x]$.
Consider the following two cases.

{\sc Case 1.} $r\ge 2$. Then $|T_j|<|T|=t$ for any integer $j$ with
$0\le j\le r-1$. It follows that there exists two integers $i_1$ and $i_2$
with $1\le i_1\ne i_2\le t$ such that
$\bar \gamma_{i_1}\ne \bar\gamma_{i_2}$, and so
$\gamma_{i_1}-\gamma_{i_2}\notin\pi\mathcal{O}_K$.
Using the induction assumption, we have for all integers
$j$ with $0\le j\le r-1$ that
\begin{equation}\label{eqno 4.7}
Z_{g_{2, j}}(s, \chi)=\dfrac{G_{2, j}(q^{-s})}{(1-q^{-1-s})
\prod\limits_{i\in T_j \atop n_i\ge 2}(1-q^{-\tilde n_i-m_i-\tilde n_im_i\gcd(n_i, m)s})},
\end{equation}
where $G_{2, j}(x)\in\mathbb{C}[x]$. From (\ref{4.4}) to (\ref{eqno 4.7}),
it follows that
$$Z_{g_1}(s, \chi)=\dfrac{G_2(q^{-s})}{(1-q^{-1-s})
\prod\limits_{i\in T \atop n_i\ge 2}(1-q^{-\tilde n_i-m_i-\tilde n_im_i\gcd(n_i, m)s})},$$
where $G_2(x)\in\mathbb{C}[x]$, as (\ref{eqno 4.3}) expected.
Thus (\ref{eqno 4.3}) is proved when $r\ge 2$.

{\sc Case 2.} $r=1$. Then $T=T_0$. So for any integer $i$
with $1\le i\le t-1$, we have $\gamma_i-\gamma_t\in\pi\mathcal{O}_K^*$.
Setting $j=0$ in (\ref{eqno 4.6}), we get that
\begin{align*}
g_{2, 0}(x, y)=\sigma x^{n_t}f_{1, 0}(x)\prod_{i=1}^{t-1}(x-\gamma_{i, 1})^{n_i}+\delta_1y^m
\end{align*}
with $\gamma_{i, 1}:=\pi^{-1}(\gamma_i-\gamma_t)\in\mathcal{O}_K^*$.

{\sc Case 2.1.} $\gamma_{i, 1}\in\pi\mathcal{O}_K$ for all integers
$i$ with $1\le i\le t-1$. Since $\sigma\in\mathcal{O}_K^{\times}$,
Lemma \ref{lem 2.6} gives us that
\begin{equation}\label{eqno 4.9}
Z_{g_{2, 0}}(s, \chi)=\dfrac{U_1(q^{-s})}
{1-q^{-1-s}}+U_2(q^{-s})Z_{g_{3, 0}}(s, \chi),
\end{equation}
where
$$g_{3, 0}(x, y):=\sigma x^{n_t}f_{1, 0}(x)
\prod_{i=1}^{t-1}(x-\gamma_{i, 2})^{n_i}+\delta_2 y^m$$
and $U_1(x), U_2(x)\in\mathbb{C}[x]$, and
$\gamma_{i_0, 2}\in\mathcal{O}_K^{\times}$,
i.e., $\gamma_{i_0, 2}\notin\pi\mathcal{O}_K$
for at least one integer $i_0$ with $1\le i_0\le t-1$.
Then $g_{3, 0}$ satisfies the assumption of Case 1.
So applying Case 1 to the polynomial $g_{3, 0}$ gives us that
\begin{equation}\label{eqno 4.10}
Z_{g_{3, 0}}(s, \chi)=\dfrac{G_{3, 0}(q^{-s})}{(1-q^{-1-s})
\prod\limits_{i=1 \atop n_i\ge 2}^t(1-q^{-\tilde n_i-m_i-\tilde n_im_i\gcd(n_i, m)s})}
\end{equation}
with $G_{3, 0}[x]\in\mathbb{C}[x]$.
Thus from (\ref{4.4}) to (\ref{eqno 4.6}) and (\ref{eqno 4.9}) to (\ref{eqno 4.10}),
we obtain that
$$Z_{g_1}(s, \chi)=\dfrac{G_3(q^{-s})}{(1-q^{-1-s})
\prod\limits_{i\in T \atop n_i\ge 2}(1-q^{-\tilde n_i-m_i-\tilde n_im_i\gcd(n_i, m)s})},$$
where $G_3(x)\in\mathbb{C}[x]$. So (\ref{eqno 4.3}) is true in this case.

{\sc Case 2.2.} $\gamma_{i_0, 1}\in\mathcal{O}_K^{\times}$
for an integer $i_0$ with $1\le i_0\le t-1$. Then $g_{2, 0}$
satisfies the assumption of Case 1. Hence
\begin{equation}\label{eqno 4.11}
Z_{g_{2, 0}}(s, \chi)=\dfrac{G_{2, 0}(q^{-s})}{(1-q^{-1-s})
\prod\limits_{i\in T_0\atop n_i\ge 2}(1-q^{-\tilde n_i-m_i-\tilde n_im_i\gcd(n_i, m)s})},
\end{equation}
It follows immediately from (\ref{4.4}) to (\ref{eqno 4.6}) and
(\ref{eqno 4.11}) that (\ref{eqno 4.3}) holds in this case.

This finishes the proof of (\ref{eqno 4.3}). So Theorem \ref{thm 1.2} is proved. \hfill$\Box$

\end{document}